\documentclass[12pt]{amsart}

\usepackage[letterpaper,margin=1.1in]{geometry}
\usepackage{amssymb} 
\usepackage{ctable} 

\newtheorem{maintheorem}{Theorem}
\newtheorem{maincor}[maintheorem]{Corollary}

\newtheorem{theorem}{Theorem}[section]
\newtheorem{lemma}[theorem]{Lemma}
\newtheorem{corollary}[theorem]{Corollary}
\newtheorem{proposition}[theorem]{Proposition}

\theoremstyle{definition}
\newtheorem{definition}[theorem]{Definition}

\theoremstyle{remark}
\newtheorem{remark}[theorem]{Remark}

\newtheorem{example}[theorem]{Example}
\newtheorem{problem}[theorem]{Problem}

\newtheorem{conjecture}[theorem]{Conjecture}
\newtheorem*{ack}{Acknowledgements}

\newcommand{\RR}{\mathbb{R}}
\newcommand{\ZZ}{\mathbb{Z}}
\newcommand{\Haus}{\mathcal{H}}
\newcommand{\Pack}{\mathcal{P}}
\newcommand{\dist}{\mathop\mathrm{dist}\nolimits}
\newcommand{\diam}{\mathop\mathrm{diam}\nolimits}
\newcommand{\side}{\mathop\mathrm{side}\nolimits}
\newcommand{\spt}{\mathop\mathrm{spt}\nolimits}
\newcommand{\tJ}{\widetilde{J}}
\newcommand{\bhat}{\hat{\beta}}
\newcommand{\hJ}{\hat{J}}
\newcommand{\two}{I\!I}
\newcommand{\Lip}{\mathop\mathrm{Lip}\nolimits}
\newcommand{\lD}[1]{\mathop{\underline{D}^{#1}}\nolimits}
\newcommand{\uD}[1]{\mathop{\overline{D}^{\,#1}}\nolimits}
\newcommand{\allDelta}{\vec{\Delta}}
\newcommand{\cT}{\mathcal{T}}
\newcommand{\res}{\hbox{ {\vrule height .22cm}{\leaders\hrule\hskip.2cm} }} 

\numberwithin{equation}{section}
\numberwithin{figure}{section}
\numberwithin{table}{section}

\begin{document}

\title[Multiscale analysis of 1-rectifiable measures]{Multiscale analysis of 1-rectifiable measures: necessary conditions}
\author{Matthew Badger \and Raanan Schul}
\thanks{M.~ Badger was partially supported by an NSF postdoctoral fellowship DMS 12-03497.
R.~ Schul was partially supported by a fellowship from the  Alfred P. Sloan Foundation as well as  by NSF  DMS 11-00008.}
\date{September 15, 2014}
\address{Department of Mathematics\\ Stony Brook University\\ Stony Brook, NY 11794-3651}
\curraddr{Department of Mathematics\\ University of Connecticut\\ Storrs, CT 06269-3009}
\email{matthew.badger@uconn.edu}
\address{Department of Mathematics\\ Stony Brook University\\ Stony Brook, NY 11794-3651}
\email{schul@math.sunysb.edu}

\begin{abstract} We repurpose tools from the theory of \emph{quantitative rectifiability} to study the \emph{qualitative rectifiability} of measures in $\RR^n$, $n\geq 2$. To each locally finite Borel measure $\mu$, we associate a function $\tJ_2(\mu,x)$ which uses a weighted sum to record how closely the mass of $\mu$ is concentrated near a line in the triples of dyadic cubes containing $x$. We show that $\tJ_2(\mu,\cdot)<\infty$ $\mu$-a.e.~ is a necessary condition for $\mu$ to give full mass to a countable family of rectifiable curves. This confirms a conjecture of Peter Jones from 2000. A novelty of this result is that no assumption is made on the upper Hausdorff density of the measure. Thus we are able to analyze general 1-rectifiable measures, including measures which are singular with respect to 1-dimensional Hausdorff measure.\end{abstract}

\maketitle

\section{Introduction}\label{s:intro}

The aim of this article is to develop a multiscale analysis of 1-rectifiable measures. Because there exist competing conventions for the terminology ``rectifiable measure" (cf.~ \cite[pp.~ 251--252]{Federer} and \cite[p.~ 228]{Mattila}), we start by specifying its meaning in the present paper. See Table \ref{t:terms} for a guide between the different conventions.

\begin{definition}[Rectifiable measure] \label{d:rect} Let $\mu$ be a Borel measure on $\RR^n$ and let $m\geq 1$ be a positive integer. We say that $\mu$ is \emph{$m$-rectifiable} if there exist countably many bounded Borel sets $E_i\subset\RR^m$ and Lipschitz maps $f_i:E_i\rightarrow\RR^n$ such that the union of the images $f_i(E_i)$ have full measure, i.e.~ $\mu\left(\RR^n\setminus \bigcup_{i=1}^\infty f_i(E_i)\right)=0$.
\end{definition}

\begin{remark} We do not require an $m$-rectifiable measure $\mu$ to be absolutely continuous with respect to the $m$-dimensional Hausdorff measure $\Haus^m$ (see \S\ref{s:prelim}). If we wish to declare this as an additional property of the measure, then we shall explicitly write $\mu\ll \Haus^m$. \end{remark}

\begin{remark} An equivalent condition for a Borel measure $\mu$ on $\RR^n$ to be 1-rectifiable is that there exist countably many \emph{rectifiable curves} $\Gamma_i\subset\RR^n$ such that $\mu(\RR^n\setminus\bigcup_i\Gamma_i)=0$. (Indeed, any Lipschitz map $f:E\rightarrow\RR^n$, $E\subset\RR^m$ extends to a global Lipschitz map $F:\RR^m\rightarrow\RR^n$. Thus, when $m=1$, one may assume without loss of generality that the sets $E_i$ in Definition \ref{d:rect} are compact intervals $[a_i,b_i]$ and take $\Gamma_i=f_i([a_i,b_i])$.) \end{remark}

\begin{table}\small\caption{Cross-reference: rectifiable measures}\label{t:terms}\begin{tabular}{ccc}
  \toprule
  this paper & \cite{Federer} & \cite{Mattila} \\ \toprule
  $\mu$ is $m$-rectifiable & $\RR^n$ is countably $(\mu,m)$ rectifiable & --- \\ \midrule
  $\mu$ is $m$-rectifiable and $\mu(\RR^n)<\infty$ & $\RR^n$ is $(\mu,m)$ rectifiable & --- \\ \midrule
  $\mu$ is $m$-rectifiable and $\mu\ll \Haus^m$ & --- & $\mu$ is $m$-rectifiable \\ \bottomrule
\end{tabular}\end{table}

The qualitative theory of rectifiable sets and absolutely continuous rectifiable measures in Euclidean spaces developed across the last century, beginning with the seminal work of Besicovitch \cite{Bes28,Bes38} and later generalized and improved upon in a series of papers by Morse and Randolph \cite{MR}, Moore \cite{Moore}, Marstrand \cite{Marstrand}, Mattila \cite{Mattila75} and Preiss \cite{Preiss}. In particular, in the presence of absolute continuity, these investigations revealed a deep connection between the rectifiability of a measure and the asymptotic behavior of the measure on small balls.

\begin{definition}[Hausdorff densities] Let $B(x,r)$ denote the closed ball in $\RR^n$ with center $x\in\RR^n$ and radius $r>0$. For each positive integer $m\geq 1$, let $\omega_m=\Haus^m(B^m(0,1))$ denote the volume of the unit ball in $\RR^m$. For all locally finite Borel measures $\mu$ on $\RR^n$, we define the \emph{lower Hausdorff $m$-density} $\lD{m}(\mu,\cdot)$ and  \emph{upper Hausdorff $m$-density} $\uD{m}(\mu,\cdot)$ by \begin{equation*} \lD{m}(\mu,x):=\liminf_{r\rightarrow 0} \frac{\mu(B(x,r))}{\omega_mr^m}\in[0,\infty]\end{equation*} and \begin{equation*} \uD{m}(\mu,x):=\limsup_{r\rightarrow 0} \frac{\mu(B(x,r))}{\omega_mr^m}\in[0,\infty]\end{equation*} for all $x\in\RR^n$.
If $\lD{m}(\mu,x)=\uD{m}(\mu,x)$ for some $x\in\RR^n$, then we write $D^m(\mu,x)$ for the common value and call $D^m(\mu,x)$ the \emph{Hausdorff $m$-density of $\mu$ at $x$}.
\end{definition}

For $\mu$ a Borel measure and $E\subset\RR^n$ a Borel set, let $\mu\res E$ denote the \emph{restriction of $\mu$ to $E$}, i.e.~ the measure defined by the rule $(\mu\res E)(F)=\mu(E\cap F)$ for all Borel $F\subset\RR^n$.

\begin{theorem}[\cite{Mattila75}]\label{t:M} Let $1\leq m\leq n-1$. Suppose $E\subset\RR^n$ is Borel and $\mu=\Haus^m\res E$ is locally finite. Then $\mu$ is $m$-rectifiable if and only if the Hausdorff $m$-density of $\mu$ exists and $D^m(\mu,x)=1$ at $\mu$-a.e.~ $x\in\RR^n$.\end{theorem}

\begin{theorem}[\cite{Preiss}]\label{t:P} Let $1\leq m\leq n-1$. If $\mu$ is a locally finite Borel measure on $\RR^n$, then $\mu$ is $m$-rectifiable and $\mu\ll \Haus^m$  if and only if the Hausdorff $m$-density of $\mu$ exists and $0<D^m(\mu,x)<\infty$ at $\mu$-a.e.~ $x\in\RR^n$.\end{theorem}

There exist additional characterizations of absolutely continuous rectifiable measures (e.g.~ in terms of the tangent measures of $\mu$). For a full survey, we refer the reader to the book \cite{Mattila} by Mattila.

In general, $m$-rectifiable measures on $\RR^n$ are not necessarily absolutely continuous with respect to Hausdorff measure $\Haus^m$. In the case $m=1$, an interesting family of singular 1-rectifiable measures was recently identified by Garnett, Killip and Schul \cite{GKS}.

\begin{example}[\cite{GKS}] \label{e:GKS} Let $h:\RR\rightarrow\RR$ be the 1-periodic function defined by \begin{equation*} h(x)=\left\{\begin{array}{ll} 2 &\text{if }x\in[\tfrac13,\tfrac23)+\mathbb{Z}\\ -1 &\text{otherwise}\end{array}\right.\!\!.\end{equation*} Given $0<\delta\leq 1/3$, define the measure $\eta$ on $\RR$ to be the weak-$*$ limit of measures $\eta_k$,\begin{equation*} d\eta_k:=\prod
_{j=0}^{k-1} [1+ (1-3\delta) h(3^jx)]\,dx.\end{equation*} Finally, assign $\mu=\eta\times\dots\times\eta$ to be the $n$-fold product of $\eta$. When $\delta=1/3$, the measure $\mu$ is Lebesgue measure on $\RR^n$. On the other hand, Garnett, Killip and Schul proved that if $\delta\leq \delta_n<1/3$ is sufficiently small, then $\mu$ is 1-rectifiable in the sense of Definition \ref{d:rect}
(see Theorem 1.1 and the paragraph ``In closing..." on p.~ 1678 of \cite{GKS}).
Therefore, for all $n\geq 2$, there exist locally finite Borel measures $\mu$ on $\RR^n$ with the following properties.
\begin{itemize}
 \item The support of $\mu$ is $\RR^n$: $\mu(B(x,r))>0$ for all $x\in\RR^n$ and for all $r>0$.
 \item The measure $\mu$ is \emph{doubling}: there is $C>1$ such that $\mu(B(x,2r))\leq C\mu(B(x,r))$ for all $x$ in the support of $\mu$ and for all $r>0$.
 \item Every Lipschitz graph (i.e.~ a set which up to an isometry of $\RR^n$ has the form $\{(x,f(x)):f:\RR^{m}\rightarrow\RR^{n-m}\text{ is Lipschitz}\}$ for $1\leq m\leq n-1$) has $\mu$ measure zero.
 \item For every connected Borel set $\Gamma\subset\RR^n$, the set $\{x\in \Gamma: \lD1(\Haus^1\res \Gamma,x)<\infty\}$ has $\mu$ measure zero.
 \item The measure $\mu$ is 1-rectifiable.
 \item The measure $\mu$ is singular with respect to  $\Haus^1$.
\end{itemize}
\end{example}

In our opinion, general rectifiable measures that are allowed to be singular with respect to Hausdorff measure are currently poorly understood. We believe that the following open problem represents a major challenge in geometric measure theory.

\begin{problem}\label{p:singular} For all $1\leq m\leq n-1$, find necessary and sufficient conditions in order for a locally finite Borel measure $\mu$ on $\RR^n$ to be $m$-rectifiable. (Do not assume that $\mu\ll \Haus^m$.)\end{problem}

In this paper, we adapt tools from the theory of quantitative rectifiability to attack Problem \ref{p:singular} in the case $m=1$. In particular, we establish new necessary conditions for a locally finite Borel measure $\mu$ on $\RR^n$ to be 1-rectifiable. To state these results, we need to introduce two concepts --- \emph{$L^2$ beta numbers} and \emph{$L^2$ Jones functions} --- which emanate from \cite{Jones-TST,DS91,DS93,BJ} (also see  \cite{Pajot96,Pajot97,Leger,Lerman,DT}).

For all $E\subset\RR^n$ and $x\in\RR^n$, let $\diam E:=\sup_{y,z\in E}|y-z|$ and $\dist(x,E):=\inf_{z\in F}|x-z|$ denote respectively the \emph{diameter of $E$} and the \emph{distance of $x$ to $E$}.

\begin{definition}[$L^2$ beta numbers] For every locally finite Borel measure $\mu$ on $\RR^n$ and every bounded Borel set $Q\subset\RR^n$ (typically we take $Q$ to be a cube), define  $\beta_2(\mu,Q)$ by
\begin{equation} \label{e:beta-2}
   \beta_2^2(\mu,Q):=\inf_\ell \int_Q \left(\frac{\dist(x,\ell)}{\diam Q}\right)^2 \frac{d\mu(x)}{\mu(Q)}\in[0,1],
\end{equation}
where $\ell$ in the infimum ranges over all lines in $\RR^n$. If $\mu(Q)=0$, then we interpret (\ref{e:beta-2}) as $\beta_2(\mu,Q)=0$.\end{definition}

The beta number $\beta_2(\mu,Q)$ records how well $\mu\res Q$ is fit by a linear regression model, in an $L^2$ sense. At one extreme, $\beta_2(\mu,Q)=0$ if and only if $\mu\res Q=\mu\res(Q\cap \ell_0)$ for some line $\ell_0$ in $\RR^n$. At the other extreme, $\beta_2(\mu,Q)\sim 1$ when the mass of $\mu\res Q$ is ``scattered" in the sense that $\mu\res Q$ assigns non-negligible mass ``far away" from every line passing through $Q$.

One may think of several natural ways to ``add up" the errors $\beta_2(\mu,Q)$ at ``all scales". Below we focus on two variations. Let $\Delta(\RR^n)$ denote the \emph{standard dyadic grid}; that is, the collection of all (closed) dyadic cubes in $\RR^n$. Also, for each cube $Q$ and $\lambda>0$, let $\lambda Q$ denote the concentric cube about $Q$ that is obtained by dilating $Q$ by a factor of $\lambda$.

\begin{definition}[$L^2$ Jones functions] Let $\mu$ be a locally finite Borel measure on $\RR^n$ and let $r>0$. The \emph{ordinary $L^2$ Jones function $J_2(\mu,r,\cdot)$ for $\mu$} is defined by \begin{equation*} J_2(\mu,r,x):=\sum_{Q}\beta_2^2(\mu,3Q)\chi_Q(x)\in[0,\infty]\end{equation*} for all $x\in\RR^n$,
where $Q$ ranges over all $Q\in\Delta(\RR^n)$ with side length at most $r$. We abbreviate the function $J_2(\mu,1,\cdot)$ starting at scale $1$ by $J_2(\mu,\cdot)$.

The \emph{density-normalized $L^2$ Jones function $\tJ_2(\mu,r,\cdot)$ for $\mu$} is defined by
\begin{equation}
\label{e:dnJones} \tJ_2(\mu,r,x):=\sum_{Q}\beta_2^2(\mu,3Q)\frac{\diam Q}{\mu(Q)}\,\chi_Q(x)\in[0,\infty]
\end{equation}
for all $x\in\RR^n$, where $Q$ ranges over all $Q\in\Delta(\RR^n)$ with side length at most $r$. (Here we take $0/0=0$.) We abbreviate the function $\tJ_2(\mu,1,\cdot)$ starting at scale 1 by $\tJ_2(\mu,\cdot)$.
\end{definition}

The ordinary $L^2$ Jones function $J_2(\mu,\cdot)$ has been used by several authors to study the rectifiability of absolutely continuous measures of the form $\mu=\Haus^1\res E$, $E\subset\RR^n$. Although we formulate Example \ref{e:ds}, Theorem \ref{t:p1} and Theorem \ref{t:p2} for 1-rectifiable measures only, analogous statements for $m$-rectifiable measures exist for all $m\geq 2$.

\begin{example}[\cite{DS91}] \label{e:ds} Let $E\subset\RR^n$ be Borel and 1-Ahlfors regular, i.e.~ suppose that there exist constants $c_1,c_2>0$ such that $c_1r\leq \Haus^1(E\cap B(x,r))\leq c_2r$ for all $x\in E$  and for all $0<r<\diam E$. If $\mu=\Haus^1\res E$ is \emph{1-uniformly rectifiable} in the sense of David and Semmes \cite{DS91,DS93}, then there exists a constant $C=C(\mu)>0$ such that $\int_Q J_2(\mu,\diam Q,\cdot)\, d\mu\leq C\diam Q$ for every dyadic cube $Q$. In particular, $J_2(\mu,\cdot)<\infty$ $\mu$-a.e.\end{example}

\begin{theorem}[{\cite[Theorem 1.1]{Pajot97}}] \label{t:p1} Suppose $K\subset\RR^n$ is compact and $\mu=\Haus^1\res K$ is finite. If both $\lD{1}(\mu,x)>0$ and $J_2(\mu,x)<\infty$ at $\mu$-a.e.~ $x\in\RR^n$, then $\mu$ is 1-rectifiable.
\end{theorem}

\begin{theorem}[{\cite[Theorem 1.2]{Pajot97}}] \label{t:p2} If $K\subset\RR^n$ is compact and 1-Ahlfors regular, then $\mu=\Haus^1\res K$ is 1-rectifiable if and only if $J_2(\mu,\cdot)<\infty$ $\mu$-a.e.\end{theorem}

By contrast, for singular rectifiable measures, $J_2(\mu,\cdot)$ can be badly behaved.

\begin{example} \label{e:GKS2} Let $\mu$ be a measure from Example \ref{e:GKS} with defining parameter $\delta\leq \delta_n$. Since $\beta_2(\mu,3Q)\sim 1$ for every dyadic cube $Q$, the ordinary Jones function $J_2(\mu,\cdot)=\infty$ $\mu$-a.e. Nevertheless, it follows from the estimates in \cite{GKS} or by Theorem \ref{t:A} below that the density-normalized Jones function $\tJ_2(\mu,\cdot)<\infty$ $\mu$-a.e.\end{example}

In 2000, Peter Jones conjectured that weighted $L^2$ Jones functions should lead to a solution of Problem \ref{p:singular} (private communication). Nam-Gyu Kang obtained unpublished results about this conjecture for measures supported on the four-corner Cantor set in $\RR^2$ (private communication). The general case is more complicated, as is evident by the measures in Example \ref{e:GKS}.

We are now ready to state our main results. Theorem \ref{t:A} confirms Jones' conjecture by connecting  the rectifiability of a measure to the pointwise behavior of its density-normalized $L^2$ Jones function. Corollary \ref{c:B} shows that for absolutely continuous measures, rectifiability of a measure also controls its ordinary $L^2$ Jones function.

\begin{maintheorem}\label{t:A} Let $\mu$ be a locally finite Borel measure on $\RR^n$. If $\mu$ is 1-rectifiable, then $\tJ_2(\mu,x)<\infty$ at $\mu$-a.e.~ $x\in\RR^n$.
\end{maintheorem}

\begin{maincor}\label{c:B} Let $\mu$ be a locally finite Borel measure on $\RR^n$. If $\mu$ is 1-rectifiable and $\mu\ll \Haus^1$, then $J_2(\mu,x)<\infty$ at $\mu$-a.e.~ $x\in\RR^n$.\end{maincor}

In the special case of 1-dimensional Hausdorff measure restricted to a compact set, Corollary \ref{c:B} (together with Lemma \ref{l:LD}) immediately yields the converse to Theorem \ref{t:p1}.

\begin{maincor}\label{c:C} Suppose that $K\subset\RR^n$ is compact and $\mu=\Haus^1\res K$ is finite. Then $\mu$ is 1-rectifiable if and only if both $\lD1(\mu,\cdot)>0$ and $J_2(\mu,\cdot)<\infty$ $\mu$-a.e.~ $x\in\RR^n$.\end{maincor}

A qualitative consequence of Theorem \ref{t:A} is that a 1-rectifiable measure $\mu$ exhibits at least one of two extreme behaviors at $\mu$-almost every $x$ in $\RR^n$: $\mu$ admits arbitrarily good 1-dimensional linear approximations near $x$ \emph{or} $\mu$ has arbitrarily large 1-dimensional density ratio near $x$. Thus, if one extreme fails, the other extreme must occur.

\begin{example} Let $\mu$ be a measure from Example \ref{e:GKS} with defining parameter $\delta\leq \delta_n$. Since $\tJ_2(\mu,x)<\infty$ and $\beta(\mu,3Q)\sim 1$ for every dyadic cube $Q$, the Hausdorff 1-density of $\mu$ exists and $D^1(\mu,x)=\infty$ at $\mu$-a.e.~ $x\in\RR^n$.
\end{example}

The proofs of Theorem \ref{t:A} and Corollary \ref{c:B} will be given in \S\ref{s:prelim} (prerequisites) and \S\ref{s:A} (main arguments). We remark that because the proofs rely on the Traveling Salesman Theorem for rectifiable curves in $\RR^n$, it is not immediately clear how to use our method to study $m$-rectifiable measures for $m\geq 2$. Our work also leaves open the possibility of finding sufficient conditions for rectifiability. Nevertheless, we believe that the idea encoded in the definition of $\tJ_2(\mu,\cdot)$---to normalize a multiscale quantity by the density of a measure scale-by-scale---is a fruitful idea that should prove useful in additional situations.

To end the paper, in \S \ref{s:related}, we discuss some connections between Theorem \ref{t:A} and Corollary \ref{c:B}, and prior work of L\'eger \cite{Leger} (Menger curvature), Lerman \cite{Lerman} (curve learning) and Tolsa \cite{Tolsa} (mass transport).

\begin{ack} The authors would like to thank Marianna Cs\"ornyei for insightful discussions about this project. The authors would also like to thank an anonymous referee for his or her careful reading of the paper. Part of this work was carried out while both authors visited the Institute for Pure and Applied Mathematics (IPAM) during the Spring 2013 long program on Interactions between Analysis and Geometry. \end{ack}

\section{Traveling Salesman Theorem and Other Prerequisites}\label{s:prelim}

In this section, we recall an essential tool from the theory of quantitative rectifiability: the Analyst's Traveling Salesman Theorem. We also collect miscellaneous lemmas which facilitate the proofs in section \ref{s:A}.

\begin{definition} Let $E\subset\RR^n$ be any set. For every bounded set $Q\subset\RR^n$ such that $E\cap Q\neq\emptyset$, define the quantity $\beta_{E}(Q)$ by \begin{equation*}
  \beta_{E}(Q):=\inf_{\ell}\sup_{x\in E\cap Q}\frac{\dist(x,\ell)}{\diam Q}\end{equation*} where $\ell$ ranges over all lines in $\RR^n$. By convention, we set $\beta_E(Q)=0$ if $E\cap Q=\emptyset$.
\end{definition}

\begin{theorem}[Traveling Salesman Theorem, \cite{Jones-TST,Ok-TST}] \label{t:tst} A bounded set $E\subset\RR^n$ is a subset of a rectifiable curve in $\RR^n$ if and only if
\begin{equation*} \beta^2(E):=\sum_{Q\in\Delta(\RR^n)}\beta_E(3Q)^{2}\diam Q<\infty.\end{equation*} Moreover, there exists a constant $C=C(n)\in(1,\infty)$ (independent of $E$) such that
\begin{itemize}
\item $\beta^2(E)\leq C \Haus^1(\Gamma)$ for every connected set $\Gamma$ containing $E$, and
\item there exists a connected set $\Gamma\supset E$ such that $\Haus^1(\Gamma)\leq C(\diam E +\beta^2(E)).$
\end{itemize} \end{theorem}

\begin{corollary} \label{c:tst} For all $n\geq 2$ and $3<a<\infty$, there is a constant $C'=C'(n,a)\in(1,\infty)$ such that if $E\subset \RR^n$ is bounded and $\Gamma$ is a connected set containing $E$, then $$\sum_{Q\in\Delta(\RR^n)} \beta_E(aQ)^2\diam Q \leq C' \Haus^1(\Gamma).$$\end{corollary}

\begin{proof} Let $n\geq 2$ and let $3<a<\infty$ be given. For any dyadic cube $Q\in\Delta(\RR^n)$ and integer $k\geq 0$, let $Q^{\uparrow k}\in\Delta(\RR^n)$ denote the $k$th ancestor of $Q$, i.e.~ $Q^{\uparrow k}$ is the unique dyadic cube containing $Q$ with $\diam Q^{\uparrow k}= 2^k\diam Q$. Choose $m\geq 1$ large enough depending only on $a$ such that $aQ\subset 3Q^{\uparrow m}$ for all $Q\in\Delta(\RR^n)$ (e.g.~ $m=\lceil \log_2 a\rceil$ will suffice.) Then $\beta_E(aQ) \leq (3\cdot 2^m/a)\beta_E(3 Q^{\uparrow m}) \leq 2^m \beta_E(3 Q^{\uparrow m})$ for all $E\subset\RR^n$ and $Q\in\Delta(\RR^n)$. Hence \begin{equation*}\begin{split} \sum_{Q\in\Delta(\RR^n)} \beta_E(aQ)^2 \diam Q &\leq 2^{2m}\sum_{Q\in\Delta(\RR^n)} \beta_E (3Q^{\uparrow m})^2 \diam Q  \\ &= 2^m\sum_{Q\in\Delta(\RR^n)} \beta_E(3Q^{\uparrow m})^2 \diam Q^{\uparrow m} \\ &= 2^{m(1+n)}\sum_{Q\in\Delta(\RR^n)} \beta_E(3Q)^2\diam Q \leq 2^{m(1+n)} C\Haus^1(\Gamma)
\end{split}\end{equation*} whenever $E\subset\RR^n$ is bounded and $\Gamma$ is a connected set containing $E$ by Theorem \ref{t:tst}. \end{proof}

\begin{remark}\label{r:hilbert} Inspecting the proof in \cite{Ok-TST} shows that the constants $C$ in Theorem \ref{t:tst} depends exponentially on the dimension $n$. Schul \cite{Schul-Hilbert} formulated a version of the Analysts' Traveling Salesman Theorem that is valid in infinite-dimensional Hilbert space. However, to obtain Theorem \ref{t:tst} with constants that are independent of the dimension $n$, one must replace the grid of dyadic cubes appearing in the definition of $\beta^2(E)$ with a ``multiresolution family" of balls that are adapted to a sequence $(X_k)_{k=1}^\infty$ of $2^{-k}$-nets for the set $E$.\end{remark}

We now enumerate several lemmas that will be used in section \ref{s:A}, starting with some facts from geometric measure theory. To fix conventions, we recall the definitions of Hausdorff and packing measures in $\RR^n$. See \cite[Chapters 4--6]{Mattila} for general background.

\begin{definition}[Hausdorff and packing measures in $\RR^n$] Let $s\geq 0$ be a real number. Let $E,E_1,E_2,\dots$ denote Borel sets in $\RR^n$. The \emph{$s$-dimensional Hausdorff measure} $\Haus^s$ is defined by $\Haus^s(E)=\lim_{\delta\rightarrow 0} \Haus^s_\delta(E)$ where \begin{equation*} \Haus^s_\delta(E)=\inf\left\{\sum_i (\diam E_i)^s:E\subset\bigcup_i E_i,\ \diam E_i\leq \delta\right\}.\end{equation*} The \emph{$s$-dimensional packing premeasure} $P^s$ is defined by $P^s(E)=\lim_{\delta\rightarrow 0} P^s_\delta(E)$ where \begin{equation*}P^s_\delta(E)=\sup\left\{\sum_{i}(2r_i)^s: x_i\in E, 2r_i\leq \delta, i\neq j\Rightarrow B(x_i,r_i)\cap B(x_j,r_j)=\emptyset\right\}.\end{equation*} The \emph{$s$-dimensional packing measure} $\Pack^s$ is defined by \begin{equation*}\Pack^s(E)=\inf\left\{\sum_i P^s(E_i):E=\bigcup_i E_i\right\}.\end{equation*}\end{definition}

\begin{lemma}\label{l:UD} Let $\mu$ be a locally finite Borel measure on $\RR^n$. Then $\mu\ll \Haus^s$ if and only if $\uD{s}(\mu,\cdot)<\infty$ $\mu$-a.e.\end{lemma}

Lemma \ref{l:UD} is Exercise 4 in \cite[Chapter 6]{Mattila}.

\begin{lemma}\label{l:LD} Let $\mu$ be a locally finite Borel measure on $\RR^n$.  If $\mu$ is $m$-rectifiable, then $\lD{m}(\mu,\cdot)>0$ $\mu$-a.e.\end{lemma}

To prove Lemma \ref{l:LD}, we first prove an auxiliary lemma.

\begin{lemma}\label{l:lpack} Let $E\subset\RR^m$. If $f:E\rightarrow\RR^n$ is Lipschitz, then $P^s(f(E))\leq (\Lip f)^s\, P^s(E)$ and  $\Pack^s(f(E))\leq (\Lip f)^s\,\Pack^s(E)$. \end{lemma}

\begin{proof}Assume that $P^s(E)<\infty$ and $f:E\rightarrow\RR^n$ is $L$-Lipschitz. Given $\varepsilon>0$, pick $\eta>0$ such that $P^s_\eta(E)\leq P^s(E)+\varepsilon$. Fix $0<\delta\leq L\eta$  and let $\{B^n(f(x_i),r_i):i\geq 1\}$ be an arbitrary disjoint collection of balls in $\RR^n$ centered in $f(E)$ such that $2r_i\leq \delta$ for all $i\geq 1$. Since $f$ is $L$-Lipschitz, \begin{equation*}f(B^m(x_i,r_i/L))\subset B^n(f(x_i),r_i)\quad\text{for all }i\geq 1.\end{equation*} Thus $\{B^m(x_i,r_i/L):i\geq 1\}$ is a disjoint collection of balls in $\RR^m$ centered in $E$ such that $2r_i/L\leq \delta/L\leq \eta$. Hence \begin{equation*} \sum_{i=1}^\infty (2r_i)^s = L^s\sum_{i=1}^\infty (2r_i/L)^s \leq L^sP^s_\eta(E)\leq L^s(P^s(E)+\varepsilon).\end{equation*} Taking the supremum over all $\delta$-packings of $f(E)$, we obtain $P^s_\delta(f(E))\leq L^s(P^s(E)+\varepsilon)$. Therefore, letting $\delta\rightarrow 0$ and $\varepsilon\rightarrow 0$, $P^s(f(E))\leq L^s P^s(E)$. The corresponding inequality for the packing measure $\Pack^s$ follows immediately from the inequality for $P^s$.\end{proof}

\begin{proof}[Proof of Lemma \ref{l:LD}] Let $\mu$ be a locally finite $m$-rectifiable Borel measure on $\RR^n$ and let $A=\{x\in\RR^n:\lD{m}(\mu,x)=0\}$. To show that $\mu(A)=0$ it suffices to prove that $\mu(A\cap f(E))=0$ for every bounded set $E\subset\RR^m$ and every Lipschitz map $f:E\rightarrow\RR^n$, since $\mu$ is $m$-rectifiable. To that end fix $E\subset\RR^m$ bounded and $f:E\rightarrow\RR^n$ Lipschitz. Then $\Pack^m(f(E))\leq (\Lip f)^m\, \Pack^m(E)<\infty$ by Lemma \ref{l:lpack} and the assumption that $E\subset\RR^m$ is bounded. Let $\lambda>0$. Since $\lD{m}(\mu,x)=0\leq\lambda$ for all $x\in A\cap f(E)$, we have \begin{equation*}\mu(A\cap f(E)) \leq \lambda \Pack^m(A\cap f(E))\leq \lambda \Pack^m(f(E))\end{equation*} by \cite[Theorem 6.11]{Mattila}. Thus, letting $\lambda\rightarrow 0$, $\mu(A\cap f(E))=0$ for every bounded set $E\subset\RR^m$ and every Lipschitz map $f:E\rightarrow\RR^n$. Therefore, $\mu(A)=0$, or equivalently, $\lD{m}(\mu,x)>0$ for $\mu$-a.e.~ $x\in\RR^n$. \end{proof}

Next we make some comparisons between the different $L^2$ Jones functions defined in the introduction.

\begin{lemma}\label{l:truncate} For every locally finite Borel measure $\mu$, the sets \begin{itemize}
 \item $\{x\in\RR^n:J_2(\mu,r,x)<\infty\}$ and $\{x\in\RR^n:J_2(\mu,r,x)=\infty\}$, and
 \item $\{x\in\RR^n:\tJ_2(\mu,r,x)<\infty\}$ and $\{x\in\RR^n:\tJ_2(\mu,r,x)=\infty\}$
\end{itemize} are independent of the parameter $r>0$.\end{lemma}

\begin{proof} With $\mu$ and $x$ fixed, changing the value of $r>0$ inserts or deletes at most a finite number of terms in the defining sums for $J_2(\mu,r,x)$ and $\tJ_2(\mu,r,x)$.\end{proof}

Below $\side Q:=\diam Q/\sqrt{n}$ denotes the side length of a cube in $\RR^n$.

\begin{lemma}\label{l:c1} Let $\mu$ be a locally finite Borel measure and let $x\in\RR^n$. If $\uD1(\mu,x)<\infty$ and $\tJ_2(\mu,x)<\infty$, then $J_2(\mu,x)<\infty$.\end{lemma}

\begin{proof} Let $\mu$ be a locally finite Borel measure on $\RR^n$, and assume that at some $x\in\RR^n$, both $\uD1(\mu,x)<\infty$ and $\tJ_2(\mu,x)<\infty$. Since $\uD1(\mu,x)<\infty$ there exist constants $M<\infty$ and $r_0>0$ such that $\mu(B(x,r))\leq Mr$ for all $0<r\leq r_0$. In particular, \begin{equation*}\mu(Q)\leq \mu(B(x,\diam Q))\leq M\diam Q\end{equation*} for every cube $Q$ containing $x$ with $\side Q\leq r_0/\sqrt{n}$ (that is, $\diam Q\leq r_0$). Hence \begin{equation}\label{e:orange}\begin{split} J_2(\mu,r_0/\sqrt{n},x) &= \sum_{\stackrel{Q\in\Delta(\RR^n)}{\side Q\leq r_0/\sqrt{n}}} \beta_2^2(\mu,3Q)\chi_Q(x)\\ &\leq M \sum_{\stackrel{Q\in\Delta(\RR^n)}{\side Q\leq r_0/\sqrt{n}}} \beta_2^2(\mu,3Q) \frac{\diam Q}{\mu(Q)}\chi_Q(x) = M \tJ_2(\mu,r_0/\sqrt{n},x).\end{split}\end{equation} But $\tJ_2(\mu,r_0/\sqrt{n},x)<\infty$ by Lemma \ref{l:truncate}, since $\tJ_2(\mu,x)<\infty$. Hence $J_2(\mu,r_0/\sqrt{n},x)<\infty$ by (\ref{e:orange}). Therefore, since $J_2(\mu,r_0/\sqrt{n},x)<\infty$, we have $J_2(\mu,x)<\infty$ by Lemma \ref{l:truncate}.
\end{proof}

\section{Proofs of the Main Results}\label{s:A}

The proof of Theorem \ref{t:A} is based on Proposition \ref{prop1L}. Roughly speaking, this proposition says that if the lower density of a finite measure $\nu$ is uniformly bounded away from 0 along a subset $E$ of a rectifiable curve $\Gamma\subset\RR^n$, then $\tJ_2(\nu,\cdot)|_E$ has finite norm in $L^1(\nu)$. In particular, $\tJ_2(\nu,x)<\infty$ at $\nu$-a.e.~ $x\in E$.

\begin{proposition}\label{prop1L} Let $\nu$ be finite Borel measure on $\RR^n$ and let $\Gamma$ be a rectifiable curve. If $E\subset\Gamma$ is Borel and there exists a constant $c_E>0$ such that $\nu(B(x,r))\geq c_E r$ for all $x\in E$ and for all $0<r\leq r_0$, then \begin{equation*} \int_{E} \tJ_2(\nu,r_0,x)\,d\nu(x)\lesssim \Haus^1(\Gamma)+\nu(\RR^n\setminus \Gamma),\end{equation*} where the implied constant depends only on $n$ and $c_E$ (see (\ref{const1L})).
\end{proposition}

\begin{proof}[Proof of Proposition \ref{prop1L}]

Let $\nu$ be a finite Borel measure on $\RR^n$ and let $\Gamma$ be a rectifiable curve in $\RR^n$. Assume that $E\subset\Gamma$ is Borel and there exists a constant $c_E>0$ such that $\nu(E\cap B(x,r))\geq c_E r$ for all $x\in E$ and for all $0<r \leq r_0$. Our objective is to find an upper bound for $\int_E\tJ_2(\nu,r_0,x)\,d\nu(x)$ in terms of the ambient dimension $n$, the lower Ahlfors regularity constant $c_E$, the length of $\Gamma$, and $\nu(\RR^n\setminus\Gamma)$.

For the duration of the proof, we let $a>3$ and $\varepsilon>0$ denote fixed constants, depending on at most $n$, which will be specified after (\ref{e:apple2}). To proceed divide the dyadic cubes $\Delta(\RR^n)$ into three subfamilies $\Delta_0$, $\Delta_\Gamma$ and $\Delta_2$, as follows:
\begin{align*}\Delta_0&=\{Q\in\Delta(\RR^n):\side Q>r_0\text{ or }\nu(E\cap Q)=0\},\\
\Delta_\Gamma&=\{Q\in\Delta(\RR^n): \side Q\leq r_0,\, \nu(E\cap Q)>0\text{ and } \varepsilon\beta_2(\nu,3Q)\leq \beta_{\Gamma}(aQ)\},\text{ and}\\
\Delta_2&=\{Q\in\Delta(\RR^n):\side Q\leq r_0,\, \nu(E\cap Q)>0\text{ and } \beta_{\Gamma}(aQ) < \varepsilon\beta_2(\nu,3Q)\}.\end{align*} Thus, the family $\Delta_0$ consists of all of the dyadic cubes in $\RR^n$ that do not contribute to $\int_E\tJ_2(\nu,r_0,x)\,d\nu(x)$. And, of the remainder, the families $\Delta_\Gamma$ and $\Delta_2$ consist of the cubes for which either $\beta_{\Gamma}(aQ)$ or $\varepsilon\beta_2(\nu,3Q)$ is the dominant quantity, respectively. Reading off the definitions of $\tJ_2(\nu,r_0,\cdot)$,  $\Delta_0$, $\Delta_\Gamma$ and $\Delta_2$, it follows that\begin{align}
\notag \int_{E}\tJ_2(\nu,r_0,x)\,d\nu(x)
 &= \sum_{Q} \beta_2^2(\nu,3Q)\frac{\diam Q}{\nu(Q)}\int_{E}\chi_Q(x)\,d\nu(x)\\
\notag &=  \sum_{Q\in\Delta_\Gamma\cup\Delta_2} \beta_2^2(\nu,3Q)\diam Q \frac{\nu(E\cap Q)}{\nu(Q)} \\
 &\leq
  \underbrace{\varepsilon^{-2}\sum_{Q\in\Delta_\Gamma}\beta_{\Gamma}^2(aQ)\diam Q}_I +
  \underbrace{\sum_{Q\in\Delta_2}\beta_2^2(\nu,3Q)\diam Q}_{\two}. \label{I-II}
\end{align} We shall estimate the terms $I$ and $\two$ separately. The former will be controlled by $\Haus^1(\Gamma)$ and the latter will be controlled by $\nu(\RR^n\setminus\Gamma)$.

To estimate $I$, we note that by Corollary \ref{c:tst},
\begin{equation}\label{Iest} I \leq \varepsilon^{-2} \sum_{Q\in\Delta(\RR^n)}\beta_\Gamma^2(aQ)\diam Q \leq C'\varepsilon^{-2}\Haus^1(\Gamma),\end{equation} where $C'$ is a finite constant determined by $n$ and $a$.

In order to estimate $\two$, decompose $\RR^n\setminus \Gamma$ into a family $\cT$ of Whitney cubes with the following specifications. \begin{itemize}
\item The union over all sets in $\cT$ is $\RR^n\setminus \Gamma$.
\item Each set in $\cT$ is a half-open cube in $\RR^n$ of the form $(a_1,b_1]\times\dots\times(a_n,b_n]$.
\item If $T_1,T_2\in\cT$, then either $T_1=T_2$ or $T_1\cap T_2=\emptyset$.
\item If $T\in\cT$, then $\dist(T,\Gamma)\leq \diam T\leq 4\dist(T,\Gamma)$.
\end{itemize} (To obtain this decomposition, modify the standard Whitney decomposition in Stein [St] by replacing each closed cube with the corresponding half-open cube.) Here $\dist(T,\Gamma)=\inf_{x\in T}\inf_{y\in \Gamma}|x-y|$. For each $k\in\ZZ$, we define \begin{equation*} \cT_k=\{T\in\cT:2^{-k-1}<\dist(T,\Gamma)\leq 2^{-k}\}.\end{equation*} Also for every cube $Q$, we set $\cT(Q)=\{T\in\cT:\nu(Q\cap T)>0\}$ and $\cT_k(Q)=\cT_k\cap\cT(Q)$. Our plan is to first estimate $\beta_2^2(\nu,3Q)\diam Q$ for each $Q\in\Delta_2$ and then estimate $\two$.

Fix $Q\in\Delta_2$, say with $\side Q=2^{-k_0}\leq r_0$. We remark that if $T\in\cT_k(3Q)$, then $k\geq k_1(k_0):=k_0-1-\lfloor \log_23\sqrt{n}\rfloor$ (to derive this, bound the distance between a point in $T\cap 3Q$ and a point in $E\cap Q$ by $\diam 3Q$). Pick any line $\ell$ in $\RR^n$ such that \begin{equation}\begin{split}\label{e:apple0} \sup_{z\in \Gamma\cap aQ} \dist(z,\ell) \leq 2\beta_{\Gamma}(aQ)\diam aQ & <2\varepsilon \beta_2(\nu,3Q) \diam aQ \\ &=(2/3)a\varepsilon \beta_2(\nu,3Q)\diam 3Q.\end{split}\end{equation} To control $\beta_2^2(\nu,3Q)$, we divide $3Q$ into two sets, $N$ (``near") and $F$ (``far"), where \begin{equation*}N=\{x\in 3Q: \dist(x,\ell)\leq (2/3)a\varepsilon\beta_2(\nu,3Q)\diam 3 Q \}\end{equation*} and \begin{equation*}F=\{x\in 3Q :\dist(x,\ell)>(2/3)a\varepsilon\beta_2(\nu,3Q)\diam 3Q\}.\end{equation*} It follows that \begin{align} \notag \beta_2^2(\nu,3Q)
 &\leq \int_N \left(\frac{\dist(x,\ell)}{\diam 3Q}\right)^2\, \frac{d\nu(x)}{\nu(3Q)} + \int_F \left(\frac{\dist(x,\ell)}{\diam 3Q}\right)^2\,\frac{d\nu(x)}{\nu(3Q)}\\
 &\leq (2/3)^2a^2\varepsilon^2\beta_2^2(\nu,3Q)+\int_F\left(\frac{\dist(x,\ell)}{\diam 3Q}\right)^2 \, \frac{d\nu(x)}{\nu(3Q)} \label{e:apple1} \\
 &\leq 3(2/3)^2a^2\varepsilon^2\beta_2^2(\nu,3Q)+2\int_F\left(\frac{\dist(x,\Gamma\cap aQ)}{\diam 3Q}\right)^2\,\frac{d\nu(x)}{\nu(3Q)}, \label{e:apple2}\end{align} where to pass between (\ref{e:apple1}) and (\ref{e:apple2}) we used the triangle inequality, $(\ref{e:apple0})$, and the inequality $(p+q)^2\leq 2p^2+2q^2$. Note that $\dist(x,\Gamma)\leq \diam 3Q=3\sqrt{n}\side Q$ for all $x\in F$, because $F\subset 3Q$ and $\Gamma\cap Q\neq\emptyset$ (since $\nu(E\cap Q)>0$ for all $Q\in\Delta_2$). Hence, specifying $a=3+6\sqrt{n}$ and $3(2/3)^2a^2\varepsilon^2=1/2$ (or generally $a\geq 3+6\sqrt{n}$ and $\varepsilon \leq 3/2a\sqrt{6}$) ensures that \begin{equation*}\dist(x,\Gamma\cap aQ)=\dist(x,\Gamma)\quad\text{for all }x\in F\end{equation*} and \begin{equation*}\beta_2^2(\nu,3Q) \leq 4\int_F\left(\frac{\dist(x,\Gamma)}{\diam 3Q}\right)^2\,\frac{d\nu(x)}{\nu(3Q)}.\end{equation*} We now employ the Whitney decomposition. Since $\Gamma\cap 3Q\subset N$  (by (\ref{e:apple0})), we have $F\subset\bigcup_{T\in\cT(3Q)}T\cap 3Q= \bigcup_{k=k_1(k_0)}^\infty\bigcup_{T\in\cT_k(3Q)}T\cap 3Q$. Note that if $T\in\cT_k(3Q)$ and $x\in T$, then $\dist (x,\Gamma)\leq \dist(T,\Gamma)+\diam T\leq 5\dist(T,\Gamma)\leq 5\cdot2^{-k}$. Thus,
\begin{equation} \label{e:golf1} \beta_2^2(\nu,3Q)\leq  100\sum_{k=k_1(k_0)}^\infty\sum_{T\in\cT_k(3Q)} \left(\frac{2^{-k}}{\diam 3Q}\right)^2\,\frac{\nu(T\cap 3Q)}{\nu(3Q)}.\end{equation} To continue, recall that $\nu(E\cap Q)>0$. Hence, we can locate $z\in E\cap Q$ and use our assumption on $E$ to conclude that \begin{equation}\label{e:golf2}
\nu(3Q)\geq \nu(B(z,2^{-k_0}))\geq c_E2^{-k_0}=\frac{c_E}{\sqrt{n}}\diam Q.\end{equation} Therefore, combining (\ref{e:golf1}) and (\ref{e:golf2}), and writing $\diam 3Q=3\diam Q$, we obtain \begin{equation}\label{beta-est}\beta_2^2(\nu,3Q)\diam Q \leq \frac{100\sqrt{n}}{9c_E}\sum_{k=k_1(k_0)}^\infty\sum_{T\in\cT_k(3Q)} \left(\frac{2^{-k}}{\diam Q}\right)^2\nu(T\cap 3Q).
\end{equation} This estimate is valid for every cube $Q\in\Delta_2$.

Equipped with (\ref{beta-est}), we can now bound $\two$. Let $l\in\ZZ$ be the smallest integer with $2^{-l}\leq r_0$. For each $k\geq l$, write $\Delta_2(k)$ for the family of cubes in $\Delta_2$ with side length $2^{-k}$. Then
\begin{align*} \two
 &= \sum_{k_0=l}^\infty \sum_{Q\in\Delta_2(k_0)} \beta_2^2(\nu,3Q)\diam Q \\
 &\leq \frac{100\sqrt{n}}{9c_E}\sum_{k_0=l}^\infty \sum_{Q\in\Delta_2(k_0)}\sum_{k=k_1(k_0)}^\infty \sum_{T\in\cT_k(3Q)}\left(
 \frac{2^{-k}}{\sqrt{n}2^{-k_0}}\right)^2\nu(T\cap 3Q).
\end{align*} Let $N_0=4^n$ denote the maximum overlap of cubes $3Q$ and $3Q'$ where the cubes $Q$ and $Q'$ are (closed) dyadic cubes of equal side length. Then
\begin{align*}
 \two &\leq \frac{100N_0\sqrt{n}}{9c_E} \sum_{k_0=l}^\infty\sum_{k=k_1(k_0)}^\infty\sum_{T\in\cT_k} \left(\frac{2^{-k}}{\sqrt{n}2^{-k_0}}\right)^2\nu(T)\\
 &= \frac{100 N_0}{9c_E\sqrt{n}}\sum_{k_0=l}^\infty\sum_{k=k_1(k_0)}^\infty (1/4)^{k-k_0}\nu\left(\bigcup\cT_k\right).
\end{align*} (Here we used our assumption that the Whitney cubes in $\cT$ are pairwise disjoint.) Next, set $m:=k_1(k_0)-k_0=1+\lfloor \log_2 3\sqrt{n}\rfloor$. Then, switching the order of summation,
\begin{align*}
 \two &\leq \frac{100N_0}{9c_E\sqrt{n}} \sum_{k=l-m}^\infty \sum_{k_0=l}^{l+k-(l-m)} (1/4)^{k-k_0}\,\nu\left(\bigcup\cT_k\right) \\
 &\leq \frac{100N_0}{9c_E\sqrt{n}}\sum_{k=l-m}^\infty \sum_{j=-m}^\infty (1/4)^j \,\nu\left(\bigcup\cT_k\right)\\
 &\leq \frac{100N_0}{9c_E\sqrt{n}}\,\frac{4^{m+1}}{3}\,\nu(\RR^n\setminus\Gamma).\end{align*}
(Once again we used the disjointness of cubes in $\cT$.) Since $4^{m+1}\leq 4^{2+\log_2 3\sqrt{n}}=16\cdot9n$ and $N_0=4^n$, it follows that \begin{equation} \label{IIest} \two\leq (\tfrac{1600}{3}\cdot 4^{n}\sqrt{n}/c_E)\nu(\RR^n\setminus\Gamma).\end{equation}  Finally, combining (\ref{I-II}), (\ref{Iest}) and (\ref{IIest}), we conclude that \begin{equation*} \int_{E}\tJ_2(\nu,r_0,x)\,d\nu(x)\lesssim \Haus^1(\Gamma)+\nu(\RR^n\setminus\Gamma),\end{equation*} where the implied constant depends only on $n$ and  $c_E$. More explicitly, \begin{equation}\label{const1L}\int_E\tJ_2(\nu,r_0,x)\,d\nu(x)\leq (8+16\sqrt{n})C'\Haus^1(\Gamma)+(\tfrac{1600}{3}\cdot4^{n}\sqrt{n}/c_E)\,\nu(\RR^n\setminus\Gamma),\end{equation} where $C'$ is the constant from Corollary \ref{c:tst} with $a=3+6\sqrt{n}$ and is exponential in the dimension $n$.\end{proof}

We are now ready to prove Theorem \ref{t:A} and Corollary \ref{c:B}.

\begin{proof}[Proof of Theorem \ref{t:A}] Let $\mu$ be a locally finite Borel measure on $\RR^n$ and assume that $\mu$ is 1-rectifiable. First, because $\mu$ is 1-rectifiable, we can find a countable family $\{\Gamma_i\}_{i=1}^\infty$ of rectifiable curves such that $\mu$ gives full mass to $\bigcup_{i=1}^\infty\Gamma_i$. Second, because $\lD1(\mu,\cdot)>0$ $\mu$-a.e.~ by Lemma \ref{l:LD}, the measure $\mu$ gives full mass to $\bigcup_{j=1}^\infty\bigcup_{k=1}^\infty E_{j,k}$, where the set $E_{j,k}=\{x\in\RR^n: \mu(B(x,r))\geq 2^{-j}r$ for all $r\in(0,2^{-k}]\}$. Thus, to establish Theorem \ref{t:A}, it suffices to prove that $\tJ_2(\mu,x)<\infty$ at $\mu$-a.e.~ $x\in\Gamma_i\cap E_{j,k}$ for all $i,j,k\geq 1$.

Suppose that $\Gamma=\Gamma_i$ and $E=\Gamma_i\cap E_{j,k}$ for some $i,j,k\geq 1$. Let $\Delta$ denote the collection of all dyadic cubes $Q$ in $\RR^n$ such that $\mu(E\cap Q)>0$ and $\side Q\leq 2^{-k}$. Define the measure \begin{equation*}\nu:=\mu\res\bigcup_{Q\in\Delta}3Q.\end{equation*} First observe that $\nu$ has bounded support, since $E$ is bounded. Hence $\nu$ is finite, because $\mu$ is locally finite. Second note that for every $x\in E$ there exists $Q_x\in \Delta$ such that $x\in Q_x$, $\side Q_x=2^{-k}$ and $B(x,2^{-k})\subset 3Q_x$. We conclude that $\nu(B(x,r))=\mu(B(x,r))\geq 2^{-j}r$ for all $r\in(0,2^{-k}]$. Thus  \begin{align*}\int_{E}\tJ_2(\mu,2^{-k},x)\,d\mu(x)
&=\sum_{Q}\beta_2^2(\mu,3Q)\frac{\diam Q}{\mu(Q)}\int_{E}\chi_Q(x) \,d\mu(x)\\
&=\sum_{Q}\beta_2^2(\nu,3Q)\frac{\diam Q}{\nu(Q)}\int_{E}\chi_Q(x) \,d\nu(x)\\
&=\int_E\tJ_2(\nu,2^{-k},x)\,d\nu(x) \lesssim \Haus^1(\Gamma)+\nu(\RR^n\setminus\Gamma)<\infty,\end{align*} by Proposition \ref{prop1L}. In particular, we have $\tJ_2(\mu,2^{-k},x)<\infty$ at $\mu$-a.e.~ $x\in E$. Therefore, $\tJ_2(\mu,x)<\infty$ at $\mu$-a.e.~ $x\in E$, by Lemma \ref{l:truncate}.
\end{proof}

\begin{proof}[Proof of the Corollary \ref{c:B}] Let $\mu$ be a locally finite Borel measure on $\RR^n$, and assume that $\mu$ is 1-rectifiable and $\mu\ll \Haus^1$. On one hand, since $\mu$ is 1-rectifiable, we have $\tJ_2(\mu,\cdot)<\infty$ $\mu$-a.e.~ by Theorem \ref{t:A}. On the other hand, since $\mu\ll \Haus^1$, we have $\uD1(\mu,\cdot)<\infty$ $\mu$-a.e.~ by Lemma \ref{l:UD}. Therefore, $J_2(\mu,\cdot)<\infty$ $\mu$-a.e.~ by Lemma \ref{l:c1}.
\end{proof}

To wrap up this section, we make two comments about variations of Proposition \ref{prop1L}, Theorem \ref{t:A} and Corollary \ref{c:B} and pose an open problem.

\begin{remark} The proof of Proposition \ref{prop1L} carries through if, rather than use equation (\ref{e:dnJones}), we defined the density-normalized Jones function $\tJ_2(\mu,r,x)$ by  \begin{equation*} \sum_{\stackrel{Q\in\Delta(\RR^n)}{\side Q\leq r}} \beta_2^2(\nu,\lambda Q)\frac{\diam Q}{\nu(Q)}\chi_Q(x)\end{equation*} for some $\lambda>1$ arbitrary. Under this scenario, the constant in (\ref{IIest}) blows up as $\lambda\rightarrow 1$.
\end{remark}

\begin{remark} In the proof of Proposition \ref{prop1L}, we assumed that the dyadic cubes used to define the density-normalized Jones function $\tJ_2(\mu,r,\cdot)$ were closed cubes.  We could have worked with open or half-open dyadic cubes instead, the only change to the proof being that $N_0=4^n$ (closed cubes) improves to $N_0=3^n$ (open or half-open cubes). Therefore, Theorem \ref{t:A} and Corollary \ref{c:B} are true independent of whether the Jones functions are defined using closed, half-open or open dyadic cubes.
\end{remark}

\begin{problem} Formulate and prove a version of Theorem \ref{t:A} in infinite-dimensional Hilbert space, or show that such a generalization cannot exist. (See Remark \ref{r:hilbert}.)
\end{problem}

\section{Related Work}
\label{s:related}

We conclude by discussing some relevant prior work. Recall that the \emph{Menger curvature} $c(x,y,z)$ of three points $x,y,z\in\RR^n$ is defined to be the inverse of the radius of the circle that passes through $x$, $y$ and $z$. If $x$, $y$ and $z$ are collinear, then $c(x,y,z)=0$. In \cite{Leger}, L\'eger proved that an integrability condition on Menger curvature is a sufficient test for certain absolutely continuous measures to be 1-rectifiable.

\begin{theorem}[\cite{Leger}] \label{t:leger} Suppose $E\subset\RR^n$ is Borel, $0<\Haus^1(E)<\infty$ and $\mu=\Haus^1\res E$. If $\iiint c^2(x,y,z)\, d\mu(x)d\mu(y)d\mu(z)<\infty$, then $\mu$ is 1-rectifiable.\end{theorem}

Placed besides one another, Theorem \ref{t:A} and Corollary \ref{c:B} (necessary conditions) and Theorem \ref{t:leger} (a sufficient condition) highlight the importance of ``curvature" to the theory of rectifiability.  For an interpretation of beta numbers as a measure of curvature, for the connection between beta numbers and Menger curvature, and for a survey of related results, we refer the reader to \cite[Chapter 3]{Pajot} and Schul \cite{Schul-AR}. Also see Hahlomaa \cite{Hahl-metric} for a version of Theorem \ref{t:leger} that is valid in metric spaces.

In \cite{Lerman}, Lerman proved that uniform control on an $L^2$ Jones function ensures that a measure gives positive mass to a rectifiable curve. Moreover, this result is quantitative. To give a precise statement of Lerman's ``$L^2$ curve learning theorem", we must introduce a variant of the ordinary Jones function, defined over shifted dyadic grids.

\begin{definition}[Shifted dyadic grids] Redefine the standard dyadic grid $\Delta(\RR^n)$ from above to be the collection of half-open dyadic cubes in $\RR^n$. For each $x\in \RR^n$, let $x+\Delta(\RR^n)$ denote the \emph{shifted dyadic grid} that is obtained by translating each cube in $\Delta(\RR^n)$ by $x$. Define $\allDelta(\RR^n)$ to be the union of the $2^n$ shifted grids $x+\Delta(\RR^n)$, $x\in\{0,1/3\}^n$.\end{definition}

\begin{definition}[$L^2$ Jones function, Lerman's variant] Let $\mu$ be a locally finite Borel measure on $\RR^n$. An \emph{$L^2$ best fit line} for $\mu$ in a cube $Q$ is any line $\ell_Q$ which achieves the minimum value of  $\int_Q\dist(x,\ell_Q)^2\,d\mu(x)$ among all lines in $\RR^n$. For every $Q\in\allDelta(\RR^n)$, define $\bhat_2(\mu,Q)$ by \begin{equation*} \bhat_2(\mu,Q):=\sup_{R}\sup_{\ell_R} \int_Q \left(\frac{\dist(x,\ell_R)}{\diam Q}\right)^2\frac{d\mu(x)}{\mu(Q)}\in[0,1]\end{equation*} where $R$ ranges over all cubes $R\in \allDelta(\RR^n)$ containing $Q$ such that \begin{equation*} 2^{j_0^*}\leq \frac{\side R}{\side Q} \leq 2^{j_1^*},\end{equation*} and $\ell_R$ ranges over all $L^2$ best fit lines for $\mu$ in $R$. Here $2\leq j_0^*\leq j_1^*$ are integer parameters. We define the \emph{modified $L^2$ Jones function} $\hJ_2(\mu,r,\cdot)$ for $\mu$  by \begin{equation*} \hJ_2(\mu,r,x)=\sum_{Q} \bhat_2(\mu,Q)^2 \chi_Q(x)\quad\text{for all }x\in\RR^n,\end{equation*} where $Q$ ranges over all cubes in $\allDelta(\RR^n)$ of side length at most $r>0$. We abbreviate the function $\hJ_2(\mu,1,\cdot)$ starting at scale 1 by $\hat J_2(\mu,\cdot)$.
\end{definition}

We now record Lerman's $L^2$ curve learning theorem. In the statement of the theorem, $\spt\mu=\{x\in\RR^n:\mu(B(x,r))>0$ for all $r>0\}$ denotes the support of $\mu$.

\begin{theorem}[{\cite[Theorem 4.8]{Lerman}}]\label{t:lerman}  Set parameters $j_0^*=2$ and $j_1^*= 2^{4+\log_2\lceil 6480e\sqrt{n}\rceil}$ ($2^{18}\sqrt{n}< j_1^*< 2^{19}\sqrt{n}$). There exist a constant $C=C(n)>1$ and an absolute constant $\lambda>1$ with the following property. If $\mu$ is a locally finite Borel measure on $\RR^n$, $Q_1\in\allDelta(\RR^n)$, and there exists $M>0$ such that \begin{equation*}\hJ_2(\mu,\lambda \side Q_1,x)\leq M\quad\text{for all }x\in\spt\mu\cap \lambda Q_1,\end{equation*} then there exists a rectifiable curve $\Gamma_1\subset \lambda Q_1$ such that $\Haus^1(\Gamma_1)\leq C e^{CM}\side Q_1$ and $\mu(\Gamma_1)\geq C^{-1}e^{-CM}\mu(Q_1)$.\end{theorem}

\begin{remark}\label{r:iterate} By iterating Theorem \ref{t:lerman}, one can show that uniform control on $\hJ_2(\mu,\cdot)$ implies that $\mu$ is 1-rectifiable. Let us describe the basic strategy.  Suppose that $\mu$ is a finite Borel measure supported on $Q_0=(0,1]^n$ with $\hat J(\mu,x)\leq M$ for all $x\in Q_0\cap\spt\mu$. Invoking Lerman's theorem once, we find a rectifiable curve $\Gamma_0$ that charges a proportion of the $\mu$ mass in $Q_0$. Next divide $Q_0\setminus \Gamma_0$ into Whitney cubes $(T_i)_{i=1}^\infty$, and invoke Lerman's theorem again on each cube $T_i$. This yields a countable family of rectifiable curves $(\Gamma_i)_{i=1}^\infty$, whose union charges a proportion of the $\mu$ mass in $Q_0\setminus \Gamma_0$. To continue, divide $Q_0\setminus \bigcup_{i=0}^\infty\Gamma_i$ into Whitney cubes\dots and so on. This  procedure yields a countable family of rectifiable curves which fully charge the mass of $\mu$. \end{remark}

Unfortunately, the proof strategy described in Remark \ref{r:iterate} cannot be used to show that ``$\hJ_2(\mu,x)<\infty$ at $\mu$-a.e.~ $x\in\RR^n$ implies $\mu$ is 1-rectifiable". On the other hand, this claim could be proved if one possessed a ``density version" of Lerman's theorem.

\begin{conjecture} \label{c:density} A version of Theorem \ref{t:lerman} holds with the hypothesis ``$\hJ_2(\mu,x)\leq M$ for all $x\in \lambda Q_1\cap \spt\mu$" replaced by  \begin{quotation}``the set $A=\{x\in \lambda Q_1: \hJ_2(\mu,x)\leq \varepsilon\}$ satisfies $\mu(A)\geq \delta\mu(\lambda Q_1)$ for some $\varepsilon\leq\varepsilon_0(n)$ and $\delta\geq \delta_0(n,\varepsilon)$"\end{quotation} and the conclusion ``$\mu(\Gamma_1)\geq C^{-1}e^{-CM}\mu(Q_1)$" replaced by \begin{quotation}``$\mu(A\cap \Gamma_1)\geq C^{-1}e^{-C\varepsilon}\mu(Q_1)$ where $C=C(n,\delta)$".\end{quotation}\end{conjecture}

\vspace{.05cm}

We believe it should be possible to verify Conjecture \ref{c:density} by rerunning the arguments used by Lerman to prove Theorem \ref{t:lerman}, but we have not checked the details.

Finally, we wish to mention a recent paper by Tolsa \cite{Tolsa}, which introduced the use of tools from the theory of mass transportation to the theory of quantitative rectifiability. In particular, Tolsa established a new characterization of uniformly rectifiable measures, expressed in terms of the $L^2$ Wasserstein distance $W_2(\cdot,\cdot)$ between probability measures. It would be interesting to know to what extent can these new tools be used to study the rectifiability of measures without an \emph{a priori} assumption of Ahlfors regularity.

\bibliography{1rect-n-refs}{}
\bibliographystyle{amsalpha}

\end{document}